\documentclass[12pt,twoside]{amsart}
\usepackage{amsmath, amsthm, amscd, amsfonts, amssymb, graphicx}
\usepackage[bookmarksnumbered, plainpages]{hyperref}

\newtheorem{thm}{Theorem}[section]
\newtheorem{cor}[thm]{Corollary}

\newtheorem{prop}[thm]{Proposition}
\newtheorem{defn}[thm]{Definition}

\newtheorem{ex}[thm]{Example}
\numberwithin{equation}{section}

\begin{document}

\title{\bf Affine connections on singular warped products}
\author{Yong Wang}

\thanks{{\scriptsize
\hskip -0.4 true cm \textit{2010 Mathematics Subject Classification:}
53C40; 53C42.
 }}

\maketitle

\begin{abstract}
 In this paper, we introduce semi-symmetric metric Koszul forms and semi-symmetric non-metric Koszul forms on singular semi-Riemannian manifolds.
 Semi-symmetric metric Koszul forms and semi-symmetric non-metric Koszul forms and their curvature of semi-regular warped products are expressed in
 terms of those of the factor manifolds. We also introduce Koszul forms associated to the almost product structure on singular almost product semi-Riemannian manifolds. Koszul forms associated to the almost product structure and their curvature of semi-regular almost product warped products are expressed in terms of those of the factor manifolds. Furthermore, we generalize the results in \cite{St2} to singular multiply warped products.
\end{abstract}

\vskip 0.2 true cm


\pagestyle{myheadings}
\markboth{\rightline {\scriptsize Wang}}
         {\leftline{\scriptsize Affine connections on singular warped products}}

\bigskip
\bigskip


\section{ Introduction}

H. A. Hayden introduced the notion of a semi-symmetric metric connection on a
Riemannian manifold \cite{HA}. K. Yano studied a Riemannian manifold endowed with
a semi-symmetric metric connection \cite{Ya}. Some properties of a Riemannian manifold
and a hypersurface of a Riemannian manifold with a semi-symmetric metric
connection were studied by T. Imai \cite{I1,I2}. Z. Nakao \cite{NA} studied submanifolds of
a Riemannian manifold with semi-symmetric metric connections. In \cite{GE},  Gozutok and Esin studied the tangent bundle of a hypersurface with semi-symmetric metric connections. In \cite{De}, Demirbag investigated the properties of a weakly Ricci symmetric manifold admitting a semi-symmetric metric connection. N. S. Agashe and M. R. Chafle introduced the notion of a semisymmetric non-metric connection and
studied some of its properties and submanifolds of a Riemannian manifold with a
semi-symmetric non-metric connection \cite{AC1,AC2}. \\
\indent The warped product provides a way to construct new semi-Riemannian manifolds
from known ones \cite{BO}. This construction has useful applications in General
Relativity, in the study of cosmological models and black holes.
In \cite{DU}, using warped product spaces, Dobarro and Unal found
 Einstein manifolds and manifolds with constant scalar curvature. In \cite{SO,Wa1,Wa2}, authors found
Einstein manifolds and manifolds with constant scalar curvature with a semi-symmetric metric connection and a kind of semi-symmetric non-metric connections. In some models of warped products,
singularities are usually present, and at such points the warping function becomes
$0$. For the Friedmann-Lemaitre-Robertson-Walker model for example, the metric
of the product manifold becomes degenerate, and the Levi-Civita connection and
Riemann curvature, as usually defined, become singular or undefined. Therefore,
we need to apply the tools of singular geometry \cite{St1}.
In \cite{St2}, Stoica studied the geometry of warped product singularities. Motivated by above works, we introduce semi-symmetric metric Koszul forms and semi-symmetric non-metric Koszul forms on singular semi-Riemannian manifolds in this paper.
 Semi-symmetric metric Koszul forms and semi-symmetric non-metric Koszul forms and their curvature of semi-regular warped products are expressed in
 terms of those of the factor manifolds.
 In \cite{ES}, Etayo and Santamaria studied some affine connections on manifolds with the product structure.
 In particular, the canonical connection for a product structure was studied. In this paper, we introduce Koszul forms associated to the almost product structure on singular almost product semi-Riemannian manifolds. Koszul forms associated to the almost product structure and their curvature of semi-regular almost product warped products are expressed in terms of those of the factor manifolds. Multiply warped products are natural generalization of warped products.
In \cite{DU}, Dobarro and Unal studied Ricci-flat and Einstein-Lorentzian multiply warped products and considered the case of having constant scalar curvature for multiply warped products and applied their results to generalized Kasner space-times. In this paper, we generalize the results in \cite{St2} to singular multiply warped products.\\
\indent In Section 2, we introduce semi-symmetric metric Koszul forms on singular semi-Riemannian manifolds.
 Semi-symmetric metric Koszul forms and their curvature of semi-regular warped products are expressed in
 terms of those of the factor manifolds.
In Section 3, we introduce semi-symmetric non-metric Koszul forms on singular semi-Riemannian manifolds.
Semi-symmetric non-metric Koszul forms and their curvature of semi-regular warped products are expressed in
 terms of those of the factor manifolds.
In Section 4, we introduce Koszul forms associated to the almost product structure on singular almost product semi-Riemannian manifolds. Koszul forms associated to the almost product structure and their curvature of semi-regular almost product warped products are expressed in
 terms of those of the factor manifolds.
In Section 5, we generalize the results in \cite{St2} to singular multiply warped products.


\vskip 1 true cm

\section{Semi-symmetric metric Koszul forms and their curvature of semi-regular warped products }

Here we omit some notations and theorems in singular geometry, for details, see \cite{St1,St2}.
Let $(M,g)$ be a singular semi-Riemannian manifold (see Definition 2.1 in \cite{St2}). Let $P$ be a vector field on $M$. Let $\mathcal{K}$ be the Koszul form on $M$ defined as $\mathcal{K}:\Gamma(TM)^3\rightarrow C^{\infty}(M)$,
\begin{equation}
\mathcal{K}(X,Y,Z):=\frac{1}{2}\left\{X\left<Y,Z\right>+Y\left<Z,X\right>-Z\left<X,Y\right>-\left<X,[Y,Z]\right>
+\left<Y,[Z,X]\right>+\left<Z,[X,Y]\right>\right\}.
\end{equation}
\begin{defn}
Semi-symmetric metric Koszul forms $\overline{\mathcal{K}_P}:\Gamma(TM)^3\rightarrow C^{\infty}(M)$ on $(M,g)$ is defined as
\begin{equation}
\overline{\mathcal{K}_P}(X,Y,Z):=\mathcal{K}(X,Y,Z)+g(Y,P)g(X,Z)-g(X,Y)g(P,Z).
\end{equation}
Usually we write $\overline{\mathcal{K}}$ instead of $\overline{\mathcal{K}_P}$.
\end{defn}

\begin{thm}
Properties of the semi-symmetric metric Koszul form of a singular semi-Riemannian manifold $(M,g)$:\\
\indent (1) Additivity and $\mathcal{R}$-linearity in each of its arguments.\\
\indent  (2) $\overline{\mathcal{K}}(fX,Y,Z)=f\overline{\mathcal{K}}(X,Y,Z).$\\
\indent (3) $\overline{\mathcal{K}}(X,fY,Z)=f\overline{\mathcal{K}}(X,Y,Z)+X(f)\left<Y,Z\right>.$\\
\indent (4)$\overline{\mathcal{K}}(X,Y,fZ)=f\overline{\mathcal{K}}(X,Y,Z).$\\
\indent (5)$\overline{\mathcal{K}}(X,Y,Z)+\overline{\mathcal{K}}(X,Z,Y)=X\left<Y,Z\right>.$\\
\indent (6)$\overline{\mathcal{K}}(X,Y,Z)-\overline{\mathcal{K}}(Y,X,Z)=\left<[X,Y]+g(Y,P)X-g(X,P)Y,Z\right>$.\\
\indent (7)$\overline{\mathcal{K}}(X,Y,Z)+\overline{\mathcal{K}}(Z,Y,X)=(L_Yg)(Z,X)+2g(Y,P)g(X,Z)-g(X,Y)g(P,Z)-g(Y,Z)g(P,X)$.\\
\indent (8)$\overline{\mathcal{K}}(X,Y,Z)+\overline{\mathcal{K}}(Y,Z,X)=Y\left<Z,X\right>+\left<[X,Y],Z\right>+g(Y,P)g(X,Z)-g(Y,Z)g(P,X)$.\\
\end{thm}
\begin{proof}
By Definition 2.1 and Theorem 2.3 in \cite{St2}, we can get this theorem. We prove (5) and (6) and other properties hold similarly.
\begin{align}
&\overline{\mathcal{K}}(X,Y,Z)+\overline{\mathcal{K}}(X,Z,Y)\\\notag
=&\mathcal{K}(X,Y,Z)+g(Y,P)g(X,Z)-g(X,Y)g(P,Z)\\\notag
&+\mathcal{K}(X,Z,Y)+g(Z,P)g(X,Y)-g(X,Z)g(P,Y)\\\notag
=&\mathcal{K}(X,Y,Z)+\mathcal{K}(X,Z,Y)\\\notag
=&X\left<Y,Z\right>.
\notag
\end{align}
\begin{align}
&\overline{\mathcal{K}}(X,Y,Z)-\overline{\mathcal{K}}(Y,X,Z)\\\notag
=&\mathcal{K}(X,Y,Z)+g(Y,P)g(X,Z)-g(X,Y)g(P,Z)\\\notag
&-\mathcal{K}(Y,X,Z)-g(X,P)g(Y,Z)+g(X,Y)g(P,Z)\\\notag
=&\left<[X,Y]+g(Y,P)X-g(X,P)Y,Z\right>.\notag
\end{align}
\end{proof}
\begin{defn}
Let $X,Y\in\Gamma(TM).$ The semi-symmetric metric lower covariant derivative of $Y$ in the direction of $X$ as the differential $1$-form
$\overline{\nabla}^{\flat}_XY\in A^1(M)$
\begin{equation}
\overline{\nabla}^{\flat}_XY(Z):=\overline{\mathcal{K}}(X,Y,Z),
\end{equation}
for any $Z\in \Gamma(TM)$.
\end{defn}
Let $\Gamma_0(TM)=\{X\in \Gamma(TM)|g(X,Y)=0, {\rm for ~any }~ Y\in \Gamma(TM)\}.$ Similarly to Corollary 5.6 in \cite{St1}, we use (5) and (6) in Theorem 2.2 and get
\begin{cor} If $X,Y\in \Gamma(TM)$ and $W\in \Gamma_0(TM)$, then
\begin{equation}
\overline{\mathcal{K}}(X,Y,W)=\overline{\mathcal{K}}(Y,X,W)=-\overline{\mathcal{K}}(X,W,Y)=-\overline{\mathcal{K}}(Y,W,X).
\end{equation}
\end{cor}
Let $Y^\flat=g(Y,\cdot)$. By Theorem 2.2 and Definition 2.3, we get
\begin{prop}$\overline{\nabla}^{\flat}_XY$ has the following properties:\\
\indent (1) Additivity and $\mathcal{R}$-linearity in each of its arguments.\\
\indent  (2) $\overline{\nabla}^{\flat}_{fX}Y=f\overline{\nabla}^{\flat}_XY.$\\
\indent (3) $\overline{\nabla}^{\flat}_X(fY)=f\overline{\nabla}^{\flat}_XY+X(f)Y^\flat.$\\
\indent (4)$(\overline{\nabla}^{\flat}_{X}Y)(Z)+(\overline{\nabla}^{\flat}_{X}Z)(Y)=X\left<Y,Z\right>.$\\
\indent (5)$\overline{\nabla}^{\flat}_{X}Y-\overline{\nabla}^{\flat}_{Y}X=[X,Y]^\flat+g(Y,P)X^\flat-g(X,P)Y^\flat.$\\
\indent (6)$(\overline{\nabla}^{\flat}_{X}Y)(Z)+(\overline{\nabla}^{\flat}_{Z}Y)(X)=(L_Yg)(Z,X)+2g(Y,P)g(X,Z)-g(X,Y)g(P,Z)-g(Y,Z)g(P,X)$.\\
\indent (7)$(\overline{\nabla}^{\flat}_{X}Y)(Z)+(\overline{\nabla}^{\flat}_{Y}Z)(X)
=Y\left<Z,X\right>+\left<[X,Y],Z\right>+g(Y,P)g(X,Z)-g(Y,Z)g(P,X)$.\\
\end{prop}
By Definition 2.1 and 2.3, we have
\begin{equation}
\overline{\nabla}^{\flat}_XY={\nabla}^{\flat}_XY+g(Y,P)X^\flat -g(X,Y)P^\flat.
\end{equation}
where $
{\nabla}^{\flat}_XY(Z):={\mathcal{K}}(X,Y,Z).$
We recall a singular manifold $(M,g)$ is radical-stationary if it satisfies the condition ${\nabla}^{\flat}_XY\in{\mathcal{A}}^{\bullet}(M)$
(for the definition of ${\mathcal{A}}^{\bullet}(M)$, see (3) in page 3 in \cite{St2}). By (2.7), we have a singular manifold $(M,g)$ is radical-stationary if and only if $\overline{\nabla}^{\flat}_XY\in{\mathcal{A}}^{\bullet}(M)$. By Corollary 2.4, if a singular manifold $(M,g)$ is radical-stationary, then for $X,Y\in \Gamma(TM)$ and $W\in \Gamma_0(TM)$, we have
\begin{equation}
\overline{\mathcal{K}}(X,Y,W)=\overline{\mathcal{K}}(Y,X,W)=-\overline{\mathcal{K}}(X,W,Y)=-\overline{\mathcal{K}}(Y,W,X)=0.
\end{equation}
\begin{defn}
Let $X\in\Gamma(TM)$, $\omega\in {\mathcal{A}}^{\bullet}(M)$, where $(M,g)$ is radical-stationary. The covariant derivative and the semi-symmetric metric covariant derivative of $\omega$ in
 the direction $X$ is defined as
 \begin{equation}{\nabla}:\Gamma(TM)\times {\mathcal{A}}^{\bullet}(M)\rightarrow A^1_d(M),~
( {\nabla}_X\omega)(Y):=X(\omega(Y))-\left<\left<{\nabla}^{\flat}_XY,\omega\right>\right>_\bullet,
\end{equation}
 \begin{equation}\overline{\nabla}:\Gamma(TM)\times {\mathcal{A}}^{\bullet}(M)\rightarrow A^1_d(M),~
( \overline{\nabla}_X\omega)(Y):=X(\omega(Y))-\left<\left<\overline{\nabla}^{\flat}_XY,\omega\right>\right>_\bullet,
\end{equation}
where for the definition of $\left<\left<,\right>\right>_\bullet$, see page 3 in \cite{St2} and $A^1_d(M)$ denotes the set of $1$-form which are smooth on the regions of constant signature.
\end{defn}
By (2.7),(2.9) and (2.10), we have
 \begin{equation}
\overline{\nabla}_X\omega={\nabla}_X\omega-\omega(X)P^\flat+\omega(P)X^\flat.
\end{equation}
Similarly to Theorem 6.13 in \cite{St1}, we have
\begin{prop}The semi-symmetric metric covariant derivative $\overline{\nabla}$ has the following properties:\\
\indent (1) Additivity and $\mathcal{R}$-linearity in each of its arguments.\\
\indent  (2) $\overline{\nabla}_{fX}\omega=f\overline{\nabla}_{X}\omega.$\\
\indent (3) $\overline{\nabla}_X(f\omega)=f\overline{\nabla}_X\omega+X(f)\omega.$\\
\indent (4)$\overline{\nabla}_{X}Y^\flat=\overline{\nabla}^\flat_{X}Y.$\\
\end{prop}
By (2.7) and (2.11) and Proposition 2.7, for a radical-stationary manifold (M,g), we have
 \begin{align}
&\overline{\nabla}_X(\overline{\nabla}^\flat_YZ)={\nabla}_X({\nabla}^\flat_YZ)+X(g(Y,Z))Y^\flat+g(Z,P){\mathcal{K}}(X,Y,\bullet)\\\notag
 &-X(g(Y,Z))P^\flat-g(Y,Z){\mathcal{K}}(X,P,\bullet)-\overline{\mathcal{K}}(Y,Z,X)P^\flat+\overline{\mathcal{K}}(Y,Z,P)X^\flat.\notag
\end{align}
 We recall
 \begin{defn}
A singular semi-Riemannian manifold $(M,g)$ satisfying $\nabla^\flat_XY\in {\mathcal{A}}^{\bullet}(M)$ and $\nabla_Z(\nabla^\flat_XY)\in {\mathcal{A}}^{\bullet}(M)$ for any $X,Y,Z\in\Gamma(TM)$ is called a semi-regular semi-Riemannian manifold.
\end{defn}
 \begin{defn}
A singular semi-Riemannian manifold $(M,g)$ satisfying $\overline{\nabla}^\flat_XY\in {\mathcal{A}}^{\bullet}(M)$ and $\overline{\nabla}_Z(\overline{\nabla}^\flat_XY)\in {\mathcal{A}}^{\bullet}(M)$ for any $X,Y,Z\in\Gamma(TM)$ is called a semi-symmetric metric  semi-regular semi-Riemannian manifold.
\end{defn}
 By (2.7) and (2.12), we have
 \begin{prop} A singular semi-Riemannian manifold $(M,g)$ is a semi-regular semi-Riemannian manifold if and only if it is a semi-symmetric metric semi-regular semi-Riemannian manifold.
\end{prop}
We recall
on a semi-regular semi-Riemannian manifold, we define the Riemann curvature tensor of the
 covariant derivative
 \begin{equation}
{R}(X,Y,Z,T):=({\nabla}_X{\nabla}^\flat_YZ)(T)-({\nabla}_Y{\nabla}^\flat_XZ)(T)
-({\nabla}^\flat_{[X,Y]}Z)(T).
\end{equation}
 On a semi-symmetric metric semi-regular semi-Riemannian manifold, we define the Riemann curvature tensor of the semi-symmetric metric
 covariant derivative
 \begin{equation}
\overline{R}(X,Y,Z,T):=(\overline{\nabla}_X\overline{\nabla}^\flat_YZ)(T)-(\overline{\nabla}_Y\overline{\nabla}^\flat_XZ)(T)
-(\overline{\nabla}^\flat_{[X,Y]}Z)(T).
\end{equation}
 Then similarly to Theorem 7.5 in \cite{St1}, we have
 \begin{thm} Let $(M,g)$ be a semi-symmetric metric semi-regular semi-Riemannian manifold, then $\overline{R}(X,Y,Z,T)$ is a smooth $(0,4)$-tensor field.
 \end{thm}
Similarly to Proposition 8.1 in \cite{St1}, we have
\begin{prop} For any vector fields $X,Y,Z,T\in\Gamma(TM)$ on a semi-symmetric metric semi-regular semi-Riemannian manifold $(M,g)$
 \begin{align}
\overline{R}(X,Y,Z,T)&=X((\overline{\nabla}^\flat_YZ)(T))-Y((\overline{\nabla}^\flat_XZ)(T))-(\overline{\nabla}^\flat_{[X,Y]}Z)(T)\\\notag
&+\left<\left<\overline{\nabla}^{\flat}_XZ,\overline{\nabla}^{\flat}_YT\right>\right>_\bullet
-\left<\left<\overline{\nabla}^{\flat}_YZ,\overline{\nabla}^{\flat}_XT\right>\right>_\bullet.
\end{align}
 \begin{align}
\overline{R}(X,Y,Z,T)&=X(\overline{\mathcal{K}}(Y,Z,T))-Y(\overline{\mathcal{K}}(X,Z,T))-\overline{\mathcal{K}}([X,Y],Z,T)\\\notag
&+\overline{\mathcal{K}}(X,Z,\bullet)\overline{\mathcal{K}}(Y,T,\bullet)-\overline{\mathcal{K}}(Y,Z,\bullet)\overline{\mathcal{K}}(X,T,\bullet).
\end{align}
\end{prop}
By (2.7),(2.12) and Theorem 2.2 and some computations, we can get\\
\begin{prop} For any vector fields $X,Y,Z,T\in\Gamma(TM)$ on a semi-symmetric metric semi-regular semi-Riemannian manifold $(M,g)$
 \begin{align}
\overline{R}(X,Y,Z,T)&={R}(X,Y,Z,T)-g(Y,P)g(X,Z)g(P,T)+g(X,P)g(Y,Z)g(P,T)\\\notag
&+\overline{\mathcal{K}}(X,P,Z)g(Y,T)-\overline{\mathcal{K}}(Y,P,Z)g(X,T)\\\notag
&-{\mathcal{K}}(X,P,T)g(Y,Z)+{\mathcal{K}}(Y,P,T)g(X,Z).
\end{align}
\end{prop}
Similarly to Proposition 7.7 (1) and (2) in \cite{St1}, we have
\begin{equation}
\overline{R}(X,Y,Z,T)=-\overline{R}(Y,X,Z,T),~~\overline{R}(X,Y,Z,T)=-\overline{R}(X,Y,T,Z).
\end{equation}
\begin{defn}(Definition 3.1 in \cite{St2})
 Let $(B,g_B)$ and $(F,g_F)$ be two singular
semi-Riemannian manifolds, and $f\in C^{\infty}(B)$ a smooth function. The warped product
of $B$ and $F$ with warping function $f$ is the semi-Riemannian manifold
\begin{equation}
B\times_fF:=(B\times F,\pi^*_B(g_B)+(f\circ \pi_B)^2\pi^*_F(g_F)),
\end{equation}
where $\pi_B: B \times F \rightarrow B $ and  $\pi_F: B \times F \rightarrow F $ are the canonical projections. It is
customary to call $B$ the base and $F$ the fiber of the warped product $B\times_fF$.
\end{defn}
\indent We know that $f$ may have zero points in above definition. The inner product on
$ B \times_ f F $ takes, for any point $p\in B\times F$ and for any pair of tangent vectors $x,y\in T_p(B\times F)$, the explicit form
\begin{equation}
\left<x,y\right>=\left<d\pi_B(x),d\pi_B(y)\right>_B+f^2(p)\left<d\pi_F(x),d\pi_F(y)\right>_F.
\end{equation}
By Proposition 3.6 in \cite{St2} and (2.2), we have
\begin{prop} Let $B\times_fF$ be a degenerate warped product and let the vector fields $X,Y,Z\in\Gamma(TB)$ and $U,V,W\in\Gamma(TF)$. Let
$\overline{\mathcal{K}}$ be the semi-symmetric metric Koszul form on $B\times_fF$ and $\overline{\mathcal{K}}_B$, $\overline{\mathcal{K}}_F$
the lifts of the semi-symmetric metric Koszul form on $B$, respectively $F$. Let $P\in \Gamma(TB)$, then\\
\indent (1) $\overline{\mathcal{K}}(X,Y,Z)=\overline{\mathcal{K}}_B(X,Y,Z)$.\\
\indent  (2) $\overline{\mathcal{K}}(X,Y,W)=\overline{\mathcal{K}}(X,W,Y)=\overline{\mathcal{K}}(W,X,Y)=0.$\\
\indent (3) $\overline{\mathcal{K}}(X,V,W)=fg_F(V,W)X(f).$\\
\indent (4)$\overline{\mathcal{K}}(V,X,W)=-\overline{\mathcal{K}}(V,W,X)=fg_F(V,W)X(f)+f^2g_B(X,P)g_F(V,W).$\\
\indent (5) $\overline{\mathcal{K}}(U,V,W)=f^2{\mathcal{K}}_F(U,V,W).$\\
\end{prop}

Similarly, we have
\begin{prop} Let $B\times_fF$ be a degenerate warped product and let the vector fields $X,Y,Z\in\Gamma(TB)$ and $U,V,W\in\Gamma(TF)$. Let
$\overline{\mathcal{K}}$ be the semi-symmetric metric Koszul form on $B\times_fF$ and $\overline{\mathcal{K}}_B$, $\overline{\mathcal{K}}_F$
the lifts of the semi-symmetric metric Koszul form on $B$, respectively $F$. Let $P\in \Gamma(TF)$, then\\
\indent (1) $\overline{\mathcal{K}}(X,Y,Z)={\mathcal{K}}_B(X,Y,Z)$.\\
\indent  (2) $\overline{\mathcal{K}}(X,Y,W)=-\overline{\mathcal{K}}(X,W,Y)=-f^2g_B(X,Y)g_F(P,W).$\\
\indent (3)$\overline{\mathcal{K}}(W,X,Y)=0.$\\
\indent (4) $\overline{\mathcal{K}}(X,V,W)=\overline{\mathcal{K}}(V,X,W)=-\overline{\mathcal{K}}(V,W,X)=fg_F(V,W)X(f).$\\
\indent (5) $\overline{\mathcal{K}}(U,V,W)=f^2{\mathcal{K}}_F(U,V,W)+f^4g_F(V,P)g_F(U,W)-f^4g_F(U,V)g_F(P,W).$\\
\end{prop}
\begin{thm} Let $(B,g_B)$ be a nondegenerate manifold and $(F,g_F)$ be a semi-regular manifold and $f\in C^{\infty}(B)$.
Then $B\times_fF$ is a semi regular warped product. Let the vector fields $X,Y,Z,T\in\Gamma(TB)$ and $U,V,W,Q\in\Gamma(TF)$ and let $H^f$ be the Hessian of $f$. Let $P\in\Gamma(TB)$, then\\
\indent (1) $\overline{R}(X,Y,Z,T)=\overline{R}_B(X,Y,Z,T)$.\\
\indent  (2) $\overline{R}(X,Y,Z,Q)=\overline{R}(Z,Q,X,Y)=0.$\\
\indent (3)$\overline{R}(X,Y,W,Q)=\overline{R}(W,Q,X,Y)=0.$\\
\indent (4) $\overline{R}(X,V,W,T)=-fH^f(X,T)g_F(V,W)+f^2g_B(X,P)g_F(V,W)g_B(P,T).$\\
\indent $-fP(f)g_F(V,W)g_B(X,T)-f^2g_B(P,P)g_F(V,W)g_B(X,T)
-f^2g_F(V,W)K_B(X,P,T).$\\
\indent (5) $\overline{R}(U,V,Z,Q)=\overline{R}(Z,Q,U,V)=0.$\\
\indent (6)  $\overline{R}(U,V,W,Q)=f^2{R}_F(U,V,W,Q)+[f^2g^*_B(df,df)-2f^3P(f)-f^4g_B(P,P)]$\\
\indent $\times[g_F(U,W)g_F(V,Q)-g_F(V,W)g_F(U,Q)].$
\end{thm}
\begin{proof}In order to prove these identities, we will use (2.17), Proposition 2.15 and Theorem 5.2 in \cite{St2}, Proposition 3.6 in \cite{St2}.
 \begin{align}
(1)\overline{R}(X,Y,Z,T)&={R}(X,Y,Z,T)-g(Y,P)g(X,Z)g(P,T)+g(X,P)g(Y,Z)g(P,T)\\\notag
&+\overline{\mathcal{K}}(X,P,Z)g(Y,T)-\overline{\mathcal{K}}(Y,P,Z)g(X,T)\\\notag
&-{\mathcal{K}}(X,P,T)g(Y,Z)+{\mathcal{K}}(Y,P,T)g(X,Z)\\\notag
&={R}_B(X,Y,Z,T)-g_B(Y,P)g_B(X,Z)g_B(P,T)+g_B(X,P)g_B(Y,Z)g_B(P,T)\\\notag
&+\overline{\mathcal{K}}_B(X,P,Z)g_B(Y,T)-\overline{\mathcal{K}}_B(Y,P,Z)g_B(X,T)\\\notag
&-{\mathcal{K}}_B(X,P,T)g_B(Y,Z)+{\mathcal{K}}_B(Y,P,T)g_B(X,Z)\\\notag
&=\overline{R}_B(X,Y,Z,T).
\end{align}
\begin{align}
(2)\overline{R}(X,Y,Z,Q)&={R}(X,Y,Z,Q)-g(Y,P)g(X,Z)g(P,Q)+g(X,P)g(Y,Z)g(P,Q)\\\notag
&+\overline{\mathcal{K}}(X,P,Z)g(Y,Q)-\overline{\mathcal{K}}(Y,P,Z)g(X,Q)\\\notag
&-{\mathcal{K}}(X,P,Q)g(Y,Z)+{\mathcal{K}}(Y,P,Q)g(X,Z)\\\notag
&=0,
\end{align}
where we applied (2) from Theorem 5.2 in \cite{St2} and (2) from Proposition 3.6 in \cite{St2}. Similarly, we can prove (2) and (3).
 \begin{align}
(4)\overline{R}(X,V,W,T)&={R}(X,V,W,T)-g(V,P)g(X,W)g(P,T)+g(X,P)g(V,W)g(P,T)\\\notag
&+\overline{\mathcal{K}}(X,P,W)g(V,T)-\overline{\mathcal{K}}(V,P,W)g(X,T)\\\notag
&-{\mathcal{K}}(X,P,T)g(V,W)+{\mathcal{K}}(V,P,T)g(X,W)\\\notag
&=-fH^f(X,T)g_F(V,W)+f^2g_B(X,P)g_F(V,W)g_B(P,T)\\\notag
&-fP(f)g_F(V,W)g_B(X,T)-f^2g_B(P,P)g_F(V,W)g_B(X,T)\\\notag
&-f^2g_F(V,W)K_B(X,P,T),
\end{align}
where we applied (5) from Theorem 5.2 in \cite{St2} and (1) from Proposition 3.6 in \cite{St2} and (4) from Proposition 2.15.
 \begin{align}
(6)\overline{R}(U,V,W,Q)&={R}(U,V,W,Q)-g(V,P)g(U,W)g(P,Q)+g(U,P)g(V,W)g(P,Q)\\\notag
&+\overline{\mathcal{K}}(U,P,W)g(V,Q)-\overline{\mathcal{K}}(V,P,W)g(U,Q)\\\notag
&-{\mathcal{K}}(U,P,Q)g(V,W)+{\mathcal{K}}(V,P,Q)g(U,W)\\\notag
&=f^2{R}_F(U,V,W,Q)+[f^2g^*_B(df,df)-2f^3P(f)-f^4g_B(P,P)]\\\notag
&\times[g_F(U,W)g_F(V,Q)-g_F(V,W)g_F(U,Q)],
\end{align}
where we applied (6) from Theorem 5.2 in \cite{St2} and (3) from Proposition 3.6 in \cite{St2} and (4) from Proposition 2.15.
\end{proof}
Using (2.17), Proposition 2.16 and Theorem 5.2 in \cite{St2}, Proposition 3.6 in \cite{St2}, similarly to Theorem 2.17, we can get
\begin{thm} Let $(B,g_B)$ be a nondegenerate manifold and $(F,g_F)$ be a semi-regular manifold and $f\in C^{\infty}(B)$.
Then $B\times_fF$ is a semi regular warped product. Let the vector fields $X,Y,Z,T\in\Gamma(TB)$ and $U,V,W,Q\in\Gamma(TF)$. Let $P\in\Gamma(TF)$, then\\
\indent (1) $\overline{R}(X,Y,Z,T)={R}_B(X,Y,Z,T)+f^2g_B(X,Z)g_F(P,P)g_B(Y,T)$\\
\indent $-f^2g_B(Y,Z)g_F(P,P)g_B(X,T)$.\\
\indent  (2) $\overline{R}(X,Y,Z,Q)=-\overline{R}(Z,Q,X,Y)=-fX(f)g_B(Y,Z)g_F(P,Q)+fY(f)g_B(X,Z)g_F(P,Q).$\\
\indent (3)$\overline{R}(X,Y,W,Q)=\overline{R}(W,Q,X,Y)=0.$\\
\indent (4) $\overline{R}(X,V,W,T)=-fH^f(X,T)g_F(V,W)-[{\mathcal{K}}_F(V,P,W)+f^2g_F(P,P)g_F(V,W)-f^2g_F(V,P)g_F(P,W)]f^2g_B(X,T).$\\
\indent (5) $\overline{R}(U,V,Z,Q)=-\overline{R}(Z,Q,U,V)=f^3Z(f)[-g_F(P,U)g_F(Q,V)+g_F(P,V)g_F(Q,U)].$\\
\indent (6)  $\overline{R}(U,V,W,Q)=f^2{R}_F(U,V,W,Q)+f^2g^*_B(df,df)[g_F(U,W)g_F(V,Q)-g_F(V,W)g_F(U,Q)]$\\
\indent ~~~~$-f^6g_F(V,P)g_F(U,W)g_F(P,Q)+f^6g_F(U,P)g_F(V,W)g_F(P,Q)$\\
\indent $+f^4{\mathcal{K}}_F(U,P,W)g_F(V,Q)-f^4{\mathcal{K}}_F(V,P,W)g_F(U,Q)$\\
\indent $-f^4{\mathcal{K}}_F(U,P,Q)g_F(V,W)
+f^4{\mathcal{K}}_F(V,P,Q)g_F(U,W)$\\
\indent $+f^6g_F(P,P)g_F(U,W)g_F(V,Q)-f^6g_F(U,P)g_F(P,W)g_F(V,Q)$\\
\indent $-f^6g_F(P,P)g_F(V,W)g_F(U,Q)+f^6g_F(V,P)g_F(P,W)g_F(U,Q).$
\end{thm}
\section{Semi-symmetric non-metric Koszul forms and their curvature of semi-regular warped products }
 \begin{defn}
 Semi-symmetric non-metric Koszul forms $\widehat{\mathcal{K}_P}:\Gamma(TM)^3\rightarrow C^{\infty}(M)$ on $(M,g)$ is defined as
\begin{equation}
\widehat{\mathcal{K}_P}(X,Y,Z):=\mathcal{K}(X,Y,Z)+g(Y,P)g(X,Z).
\end{equation}
Usually we write $\widehat{\mathcal{K}}$ instead of $\widehat{\mathcal{K}_P}$.
\end{defn}
By Definition 3.1 and Theorem 2.3 in \cite{St2}, we can get
\begin{thm}
Properties of the semi-symmetric non-metric Koszul form of a singular semi-Riemannian manifold $(M,g)$:\\
\indent (1) Additivity and $\mathcal{R}$-linearity in each of its arguments.\\
\indent  (2) $\widehat{\mathcal{K}}(fX,Y,Z)=f\widehat{\mathcal{K}}(X,Y,Z).$\\
\indent (3) $\widehat{\mathcal{K}}(X,fY,Z)=f\widehat{\mathcal{K}}(X,Y,Z)+X(f)\left<Y,Z\right>.$\\
\indent (4)$\widehat{\mathcal{K}}(X,Y,fZ)=f\widehat{\mathcal{K}}(X,Y,Z).$\\
\indent (5)$\widehat{\mathcal{K}}(X,Y,Z)+\widehat{\mathcal{K}}(X,Z,Y)=X\left<Y,Z\right>+g(Y,P)g(X,Z)+g(Z,P)g(X,Y).$\\
\indent (6)$\widehat{\mathcal{K}}(X,Y,Z)-\widehat{\mathcal{K}}(Y,X,Z)=\left<[X,Y]+g(Y,P)X-g(X,P)Y,Z\right>$.\\
\indent (7)$\widehat{\mathcal{K}}(X,Y,Z)+\widehat{\mathcal{K}}(Z,Y,X)=(L_Yg)(Z,X)+2g(Y,P)g(X,Z)$.\\
\indent (8)$\widehat{\mathcal{K}}(X,Y,Z)+\widehat{\mathcal{K}}(Y,Z,X)=Y\left<Z,X\right>+\left<[X,Y],Z\right>+g(Y,P)g(X,Z)+g(Z,P)g(X,Y)$.\\
\end{thm}
\begin{defn}
Let $X,Y\in\Gamma(TM).$ The semi-symmetric non-metric lower covariant derivative of $Y$ in the direction of $X$ as the differential $1$-form
$\widehat{\nabla}^{\flat}_XY\in A^1(M)$
\begin{equation}
\widehat{\nabla}^{\flat}_XY(Z):=\widehat{\mathcal{K}}(X,Y,Z),
\end{equation}
for any $Z\in \Gamma(TM)$.
\end{defn}
 Similarly to Corollary 5.6 in \cite{St1}, we use (5) and (6) in Theorem 3.2 and get
\begin{cor} If $X,Y\in \Gamma(TM)$ and $W\in \Gamma_0(TM)$, then
\begin{equation}
\widehat{\mathcal{K}}(X,Y,W)=\widehat{\mathcal{K}}(Y,X,W)=-\widehat{\mathcal{K}}(X,W,Y)=-\widehat{\mathcal{K}}(Y,W,X).
\end{equation}
\end{cor}
 By Theorem 3.2 and Definition 3.3, we get
\begin{prop}$\widehat{\nabla}^{\flat}_XY$ has the following properties:\\
\indent (1) Additivity and $\mathcal{R}$-linearity in each of its arguments.\\
\indent  (2) $\widehat{\nabla}^{\flat}_{fX}Y=f\widehat{\nabla}^{\flat}_XY.$\\
\indent (3) $\widehat{\nabla}^{\flat}_X(fY)=f\widehat{\nabla}^{\flat}_XY+X(f)Y^\flat.$\\
\indent (4)$(\widehat{\nabla}^{\flat}_{X}Y)(Z)+(\widehat{\nabla}^{\flat}_{X}Z)(Y)=X\left<Y,Z\right>+g(Y,P)g(X,Z)+g(Z,P)g(X,Y).$\\
\indent (5)$\widehat{\nabla}^{\flat}_{X}Y-\widehat{\nabla}^{\flat}_{Y}X=[X,Y]^\flat+g(Y,P)X^\flat-g(X,P)Y^\flat.$\\
\indent (6)$(\widehat{\nabla}^{\flat}_{X}Y)(Z)+(\widehat{\nabla}^{\flat}_{Z}Y)(X)=(L_Yg)(Z,X)+2g(Y,P)g(X,Z)$.\\
\indent (7)$(\widehat{\nabla}^{\flat}_{X}Y)(Z)+(\widehat{\nabla}^{\flat}_{Y}Z)(X)
=Y\left<Z,X\right>+\left<[X,Y],Z\right>+g(Y,P)g(X,Z)+g(Z,P)g(X,Y)$.\\
\end{prop}
By Definition 3.1 and 3.3, we have
\begin{equation}
\widehat{\nabla}^{\flat}_XY={\nabla}^{\flat}_XY+g(Y,P)X^\flat .
\end{equation}
By (3.4), we have a singular manifold $(M,g)$ is radical-stationary if and only if $\widehat{\nabla}^{\flat}_XY\in{\mathcal{A}}^{\bullet}(M)$. By Corollary 3.4, if a singular manifold $(M,g)$ is radical-stationary, then for $X,Y\in \Gamma(TM)$ and $W\in \Gamma_0(TM)$, we have
\begin{equation}
\widehat{\mathcal{K}}(X,Y,W)=\widehat{\mathcal{K}}(Y,X,W)=-\widehat{\mathcal{K}}(X,W,Y)=-\widehat{\mathcal{K}}(Y,W,X)=0.
\end{equation}
\begin{defn}
Let $X\in\Gamma(TM)$, $\omega\in {\mathcal{A}}^{\bullet}(M)$, where $(M,g)$ is radical-stationary. The semi-symmetric non-metric covariant derivative of $\omega$ in
 the direction $X$ is defined as
 \begin{equation}\widehat{\nabla}:\Gamma(TM)\times {\mathcal{A}}^{\bullet}(M)\rightarrow A^1_d(M),
( \widehat{\nabla}_X\omega)(Y):=X(\omega(Y))-\left<\left<\widehat{\nabla}^{\flat}_XY,\omega\right>\right>_\bullet.
\end{equation}
\end{defn}
By (3.4),(3.6) and (2.9), we have
 \begin{equation}
\widehat{\nabla}_X\omega={\nabla}_X\omega-\omega(X)P^\flat.
\end{equation}
Similarly to Theorem 6.13 in \cite{St1}, we have
\begin{prop}The semi-symmetric non-metric covariant derivative $\widehat{\nabla}$ has the following properties:\\
\indent (1) Additivity and $\mathcal{R}$-linearity in each of its arguments.\\
\indent  (2) $\widehat{\nabla}_{fX}\omega=f\widehat{\nabla}_{X}\omega.$\\
\indent (3) $\widehat{\nabla}_X(f\omega)=f\widehat{\nabla}_X\omega+X(f)\omega.$\\
\indent (4)$\widehat{\nabla}_{X}Y^\flat=\widehat{\nabla}^\flat_{X}Y-g(Y,P)X^\flat-g(X,Y)P^\flat.$\\
\end{prop}
By (3.4) and (3.7) and Theorem 6.13 in \cite{St1}, for a radical-stationary manifold (M,g), we have
 \begin{align}
&\widehat{\nabla}_X(\widehat{\nabla}^\flat_YZ)={\nabla}_X({\nabla}^\flat_YZ)+X(g(Y,Z))Y^\flat\\\notag
&+g(Z,P){\mathcal{K}}(X,Y,\bullet)-\overline{\mathcal{K}}(Y,Z,X)P^\flat.\notag
\end{align}
 \begin{defn}
A singular semi-Riemannian manifold $(M,g)$ satisfying $\widehat{\nabla}^\flat_XY\in {\mathcal{A}}^{\bullet}(M)$ and $\widehat{\nabla}_Z(\widehat{\nabla}^\flat_XY)\in {\mathcal{A}}^{\bullet}(M)$ for any $X,Y,Z\in\Gamma(TM)$ is called a semi-symmetric non-metric  semi-regular semi-Riemannian manifold.
\end{defn}
By (3.4) and (3.8), we have
 \begin{prop} A singular semi-Riemannian manifold $(M,g)$ is a semi-regular semi-Riemannian manifold if and only if it is a semi-symmetric non-metric semi-regular semi-Riemannian manifold.
\end{prop}
On a semi-symmetric non-metric semi-regular semi-Riemannian manifold, we define the Riemann curvature tensor of the semi-symmetric non-metric
 covariant derivative
 \begin{equation}
\widehat{R}(X,Y,Z,T):=(\widehat{\nabla}_X\widehat{\nabla}^\flat_YZ)(T)-(\widehat{\nabla}_Y\widehat{\nabla}^\flat_XZ)(T)
-(\widehat{\nabla}^\flat_{[X,Y]}Z)(T).
\end{equation}
 Then we have
 \begin{thm} Let $(M,g)$ be a semi-symmetric non-metric semi-regular semi-Riemannian manifold, then $\widehat{R}(X,Y,Z,T)$ is smooth and $C^{\infty}(M)$-linear about $X,Y,T$ and
 \begin{align}
\widehat{R}(X,Y,fZ,T)&=f\widehat{R}(X,Y,Z,T)+X(f)g(Z,P)g(Y,T)
+X(f)g(Y,Z)g(P,T)\\\notag
&-Y(f)g(Z,P)g(X,T)-Y(f)g(X,Z)g(P,T).
\end{align}
 \end{thm}
\begin{proof}
\begin{align}
 \widehat{R}(X,Y,fZ,T)&=(\widehat{\nabla}_X\widehat{\nabla}^\flat_YfZ)(T)-(\widehat{\nabla}_Y\widehat{\nabla}^\flat_XfZ)(T)
-(\widehat{\nabla}^\flat_{[X,Y]}fZ)(T)\\\notag
&=\widehat{\nabla}_X(f\widehat{\nabla}^\flat_YZ+(Yf)Z^\flat)(T)-\widehat{\nabla}_Y(f\widehat{\nabla}^\flat_XZ+(Xf)Z^\flat)(T)\\\notag
&-(f\widehat{\nabla}^\flat_{[X,Y]}Z+[X,Y](f)Z^\flat)(T)\\\notag
&=f\widehat{R}(X,Y,Z,T)+X(f)(\widehat{\nabla}^\flat_YZ-\widehat{\nabla}_YZ^\flat)(T)-Y(f)(\widehat{\nabla}^\flat_XZ-\widehat{\nabla}_XZ^\flat)(T)\\
&=f\widehat{R}(X,Y,Z,T)+X(f)g(Z,P)g(Y,T)
+X(f)g(Y,Z)g(P,T)\\\notag
&-Y(f)g(Z,P)g(X,T)-Y(f)g(X,Z)g(P,T),
\notag
\end{align}
 where we applied Proposition 3.5 and Proposition 3.7.
 \end{proof}
 \indent Similarly to Proposition 8.1 in \cite{St1}, we have
\begin{prop} For any vector fields $X,Y,Z,T\in\Gamma(TM)$ on a semi-symmetric non-metric semi-regular semi-Riemannian manifold $(M,g)$
 \begin{align}
\widehat{R}(X,Y,Z,T)&=X((\widehat{\nabla}^\flat_YZ)(T))-Y((\widehat{\nabla}^\flat_XZ)(T))-(\widehat{\nabla}^\flat_{[X,Y]}Z)(T)\\\notag
&+\left<\left<\widehat{\nabla}^{\flat}_XZ,\widehat{\nabla}^{\flat}_YT\right>\right>_\bullet
-\left<\left<\widehat{\nabla}^{\flat}_YZ,\widehat{\nabla}^{\flat}_XT\right>\right>_\bullet.
\end{align}
 \begin{align}
\widehat{R}(X,Y,Z,T)&=X(\widehat{\mathcal{K}}(Y,Z,T))-Y(\widehat{\mathcal{K}}(X,Z,T))-\widehat{\mathcal{K}}([X,Y],Z,T)\\\notag
&+\widehat{\mathcal{K}}(X,Z,\bullet)\widehat{\mathcal{K}}(Y,T,\bullet)-\widehat{\mathcal{K}}(Y,Z,\bullet)\widehat{\mathcal{K}}(X,T,\bullet).
\end{align}
\end{prop}
By (3.4),(3.8),(3.9) and (6) from Theorem 2.3 in \cite{St2} and some computations, we can get\\
\begin{prop} For any vector fields $X,Y,Z,T\in\Gamma(TM)$ on a semi-symmetric non-metric semi-regular semi-Riemannian manifold $(M,g)$
 \begin{align}
\widehat{R}(X,Y,Z,T)&={R}(X,Y,Z,T)+X(g(Z,P))g(Y,T)-Y(g(Z,P))g(X,T)\\\notag
&-\widehat{\mathcal{K}}(Y,Z,X)g(P,T)+\widehat{\mathcal{K}}(X,Z,Y)g(P,T).
\end{align}
\end{prop}
We have
\begin{equation}
\widehat{R}(X,Y,Z,T)=-\widehat{R}(Y,X,Z,T),~~\widehat{R}(X,Y,Z,T)\neq-\widehat{R}(X,Y,T,Z).
\end{equation}
By Proposition 3.6 in \cite{St2} and (3.1), we have
\begin{prop} Let $B\times_fF$ be a degenerate warped product and let the vector fields $X,Y,Z\in\Gamma(TB)$ and $U,V,W\in\Gamma(TF)$. Let
$\widehat{\mathcal{K}}$ be the semi-symmetric non-metric Koszul form on $B\times_fF$ and $\widehat{\mathcal{K}}_B$, $\widehat{\mathcal{K}}_F$
the lifts of the semi-symmetric non-metric Koszul form on $B$, respectively $F$. Let $P\in \Gamma(TB)$, then\\
\indent (1) $\widehat{\mathcal{K}}(X,Y,Z)=\widehat{\mathcal{K}}_B(X,Y,Z)$.\\
\indent  (2) $\widehat{\mathcal{K}}(X,Y,W)=\widehat{\mathcal{K}}(X,W,Y)=\widehat{\mathcal{K}}(W,X,Y)=0.$\\
\indent (3) $\widehat{\mathcal{K}}(X,V,W)=-\widehat{\mathcal{K}}(V,W,X)=fg_F(V,W)X(f).$\\
\indent (4)$\widehat{\mathcal{K}}(V,X,W)=fg_F(V,W)X(f)+f^2g_B(X,P)g_F(V,W).$\\
\indent (5) $\widehat{\mathcal{K}}(U,V,W)=f^2{\mathcal{K}}_F(U,V,W).$\\
\end{prop}
Similarly, we have
\begin{prop} Let $B\times_fF$ be a degenerate warped product and let the vector fields $X,Y,Z\in\Gamma(TB)$ and $U,V,W\in\Gamma(TF)$. Let
$\widehat{\mathcal{K}}$ be the semi-symmetric non-metric Koszul form on $B\times_fF$ and $\widehat{\mathcal{K}}_B$, $\widehat{\mathcal{K}}_F$
the lifts of the semi-symmetric non-metric Koszul form on $B$, respectively $F$. Let $P\in \Gamma(TF)$, then\\
\indent (1) $\widehat{\mathcal{K}}(X,Y,Z)={\mathcal{K}}_B(X,Y,Z)$.\\
\indent  (2) $\widehat{\mathcal{K}}(X,Y,W)=\widehat{\mathcal{K}}(W,X,Y)=0.$\\
\indent (3)$\widehat{\mathcal{K}}(X,W,Y)=f^2g_B(X,Y)g_F(P,W).$\\
\indent (4) $\widehat{\mathcal{K}}(X,V,W)=\widehat{\mathcal{K}}(V,X,W)=-\widehat{\mathcal{K}}(V,W,X)=fg_F(V,W)X(f).$\\
\indent (5) $\widehat{\mathcal{K}}(U,V,W)=f^2{\mathcal{K}}_F(U,V,W)+f^4g_F(V,P)g_F(U,W).$\\
\end{prop}
Using Proposition 3.12-3.14 and Theorem 5.2 in \cite{St2}, we have

\begin{thm} Let $(B,g_B)$ be a nondegenerate manifold and $(F,g_F)$ be a semi-regular manifold and $f\in C^{\infty}(B)$.
Then $B\times_fF$ is a semi regular warped product. Let the vector fields $X,Y,Z,T\in\Gamma(TB)$ and $U,V,W,Q\in\Gamma(TF)$ and let $H^f$ be the Hessian of $f$. Let $P\in\Gamma(TB)$, then\\
\indent (1) $\widehat{R}(X,Y,Z,T)=\widehat{R}_B(X,Y,Z,T)$.\\
\indent  (2) $\widehat{R}(X,Y,Z,Q)=\widehat{R}(Z,Q,X,Y)=\widehat{R}(X,Y,Q,Z)=0.$\\
\indent (3) $\widehat{R}(X,Y,W,Q)=\widehat{R}(W,Q,X,Y)=0.$\\
\indent (4) $\widehat{R}(X,V,W,T)=-fH^f(X,T)g_F(V,W)+2fX(f)g_F(V,W)g_B(P,T).$\\
\indent (5) $\widehat{R}(X,V,T,W)=fH^f(X,T)g_F(V,W)+f^2g_F(V,W)X(g_B(P,T)).$\\
\indent (6) $\widehat{R}(U,V,Z,Q)=\widehat{R}(Z,Q,U,V)=0.$\\
\indent (7) $\widehat{R}(U,V,Q,Z)=f^2g_B(P,Z)[\mathcal{K}_F(U,Q,V)-\mathcal{K}_F(V,Q,U)].$\\
\indent (8)  $\widehat{R}(U,V,W,Q)=f^2{R}_F(U,V,W,Q)+f^2g^*_B(df,df)
[g_F(U,W)g_F(V,Q)-g_F(V,W)g_F(U,Q)].$
\end{thm}
\begin{thm} Let $(B,g_B)$ be a nondegenerate manifold and $(F,g_F)$ be a semi-regular manifold and $f\in C^{\infty}(B)$.
Then $B\times_fF$ is a semi regular warped product. Let the vector fields $X,Y,Z,T\in\Gamma(TB)$ and $U,V,W,Q\in\Gamma(TF)$. Let $P\in\Gamma(TF)$, then\\
\indent (1) $\widehat{R}(X,Y,Z,T)={R}_B(X,Y,Z,T)$.\\
\indent  (2) $\widehat{R}(X,Y,Z,Q)=f^2g_F(P,Q)[\mathcal{K}_B(X,Z,Y)-\mathcal{K}_B(Y,Z,X)].$\\
\indent (3) $\widehat{R}(Z,Q,X,Y)=\widehat{R}(X,Y,Q,Z)=0.$\\
\indent (4) $\widehat{R}(X,Y,W,Q)=\widehat{R}(W,Q,X,Y)=0.$\\
\indent (5) $\widehat{R}(X,V,W,T)=-fH^f(X,T)g_F(V,W)-f^2V(g_F(W,P))g_B(X,T).$\\
\indent (6) $\widehat{R}(X,V,T,W)=fH^f(X,T)g_F(V,W).$\\
\indent (7) $\widehat{R}(U,V,Z,Q)=\widehat{R}(U,V,Q,Z)=0.$\\
\indent (8) $\widehat{R}(Z,Q,U,V)=f^3Z(f)g_F(Q,U)g_F(P,V).$\\
\indent (9)  $\widehat{R}(U,V,W,Q)=f^2{R}_F(U,V,W,Q)+f^2g^*_B(df,df)
[g_F(U,W)g_F(V,Q)-g_F(V,W)g_F(U,Q)]$\\
\indent $+f^4g_F(P,Q)[\mathcal{K}_F(U,W,V)-\mathcal{K}_F(V,W,U)]+f^4U(g_F(W,P))g_F(V,Q)-f^4V(g_F(W,P))g_F(U,Q).$\\
\end{thm}
\section{Koszul forms associated to the almost product structure and their curvature of semi-regular almost product warped products}
 \begin{defn}
A singular semi-Riemannian manifold $(M,g)$ is called a almost product singular semi-Riemannian manifold if there exists a almost product
structure $J:TM\rightarrow TM$ satisfying $J^2=id$ and $g(JX,JY)=g(X,Y)$ for any $X,Y\in\Gamma(TM)$.
\end{defn}
When $(M,g,J)$ is a non-degenerate almost product manifolds, we define the canonical connection $\widetilde{\nabla}_XY=\frac{1}{2}\nabla_XY+\frac{1}{2}J\nabla_XJY$ where $\nabla$ is the Levi-Civita connection. Then the canonical connection
satisfies $\widetilde{\nabla}(g)=0$ and  $\widetilde{\nabla}(J)=0$. Then
\begin{equation}
g(\widetilde{\nabla}_XY,Z)=\frac{1}{2}g(\nabla_XY,Z)+\frac{1}{2}g(J\nabla_XJY,Z)=\frac{1}{2}g(\nabla_XY,Z)+\frac{1}{2}g(\nabla_XJY,JZ).
\end{equation}
By (4.1), we define
\begin{defn}
Almost product Koszul forms $\widetilde{\mathcal{K}}:\Gamma(TM)^3\rightarrow C^{\infty}(M)$ on $(M,g,J)$ is defined as
\begin{equation}
\widetilde{\mathcal{K}}(X,Y,Z):=\frac{1}{2}[\mathcal{K}(X,Y,Z)+\mathcal{K}(X,JY,JZ)].
\end{equation}
\end{defn}
\begin{thm}
Properties of the almost product Koszul form of a singular almost product semi-Riemannian manifold $(M,g)$:\\
\indent (1) Additivity and $\mathcal{R}$-linearity in each of its arguments.\\
\indent  (2) $\widetilde{\mathcal{K}}(fX,Y,Z)=f\widetilde{\mathcal{K}}(X,Y,Z).$\\
\indent (3) $\widetilde{\mathcal{K}}(X,fY,Z)=f\widetilde{\mathcal{K}}(X,Y,Z)+X(f)\left<Y,Z\right>.$\\
\indent (4)$\widetilde{\mathcal{K}}(X,Y,fZ)=f\widetilde{\mathcal{K}}(X,Y,Z).$\\
\indent (5)$\widetilde{\mathcal{K}}(X,Y,Z)+\widetilde{\mathcal{K}}(X,Z,Y)=X\left<Y,Z\right>.$\\
\indent (6)$\widetilde{\mathcal{K}}(X,Y,Z)-\widetilde{\mathcal{K}}(Y,X,Z)=\frac{1}{2}\left<[X,Y],Z\right>
+\frac{1}{2}[\mathcal{K}(X,JY,JZ)-\mathcal{K}(Y,JX,JZ)]$.\\
\indent (7)$\widetilde{\mathcal{K}}(X,JY,JZ)=\widetilde{\mathcal{K}}(X,Y,Z)$.\\
\end{thm}
\begin{proof}
By Definition 4.2 and Theorem 2.3 in \cite{St2}, we can get this theorem. We prove (5) and (6) and other properties hold similarly.
\begin{align}
&\widetilde{\mathcal{K}}(X,Y,Z)+\widetilde{\mathcal{K}}(X,Z,Y)\\\notag
=&\frac{1}{2}[\mathcal{K}(X,Y,Z)+\mathcal{K}(X,JY,JZ)+\mathcal{K}(X,Z,Y)+\mathcal{K}(X,JZ,JY)]\\\notag
=&\frac{1}{2}[X(g(Y,Z))+X(g(JZ,JY))]\\\notag
=&X(g(Y,Z)).
\notag
\end{align}
\begin{align}
&\widetilde{\mathcal{K}}(X,Y,Z)-\widetilde{\mathcal{K}}(Y,X,Z)\\\notag
=&\frac{1}{2}[\mathcal{K}(X,Y,Z)-\mathcal{K}(Y,X,Z)+\mathcal{K}(X,JY,JZ)-\mathcal{K}(Y,JX,JZ)]\\\notag
=&\frac{1}{2}\left<[X,Y],Z\right>
+\frac{1}{2}[\mathcal{K}(X,JY,JZ)-\mathcal{K}(Y,JX,JZ)].
\end{align}
\end{proof}
\begin{defn}
Let $X,Y\in\Gamma(TM).$ The almost product lower covariant derivative of $Y$ in the direction of $X$ as the differential $1$-form
$\widetilde{\nabla}^{\flat}_XY\in A^1(M)$
\begin{equation}
\widetilde{\nabla}^{\flat}_XY(Z):=\widetilde{\mathcal{K}}(X,Y,Z),
\end{equation}
for any $Z\in \Gamma(TM)$.
\end{defn}
By Theorem 4.3 and Definition 4.4, we get
\begin{prop}$\widetilde{\nabla}^{\flat}_XY$ has the following properties:\\
\indent (1) Additivity and $\mathcal{R}$-linearity in each of its arguments.\\
\indent  (2) $\widetilde{\nabla}^{\flat}_{fX}Y=f\widetilde{\nabla}^{\flat}_XY.$\\
\indent (3) $\widetilde{\nabla}^{\flat}_X(fY)=f\widetilde{\nabla}^{\flat}_XY+X(f)Y^\flat.$\\
\indent (4)$(\widetilde{\nabla}^{\flat}_{X}Y)(Z)+(\widetilde{\nabla}^{\flat}_{X}Z)(Y)=X\left<Y,Z\right>.$\\
\indent (5)$\widetilde{\nabla}^{\flat}_{X}Y(Z)-\widetilde{\nabla}^{\flat}_{Y}X(Z)=\frac{1}{2}\left<[X,Y],Z\right>
+\frac{1}{2}[\mathcal{K}(X,JY,JZ)-\mathcal{K}(Y,JX,JZ)]$\\
\indent (6)$(\widetilde{\nabla}^{\flat}_{X}(JY))(JZ)=(\widetilde{\nabla}^{\flat}_{X}Y)(Z)$.\\
\end{prop}

\begin{prop}
A singular almost product manifold $(M,g,J)$ is radical-stationary, then $\widetilde{\nabla}^{\flat}_XY\in{\mathcal{A}}^{\bullet}(M)$
\end{prop}
\begin{proof}
Let $Z\in\Gamma_0(TM)$ and $W\in\Gamma(TM)$. By $g(JZ,W)=g(J^2Z,JW)=g(Z,JW)=0$, then $JZ\in\Gamma_0(TM)$. Since $(M,g,J)$ is radical-stationary
then $\mathcal{K}(X,Y,Z)=0$ for $X,Y\in\Gamma(TM)$ and $Z\in\Gamma_0(TM)$. So by Definition 4.2, we have $\widetilde{\mathcal{K}}(X,Y,Z)=0$ for $X,Y\in\Gamma(TM)$ and $Z\in\Gamma_0(TM)$. By Definition 4.4, we prove this proposition.
\end{proof}
\begin{defn}
Let $X\in\Gamma(TM)$, $\omega\in {\mathcal{A}}^{\bullet}(M)$, where $(M,g)$ is radical-stationary. The almost product covariant derivative of $\omega$ in
 the direction $X$ is defined as
\begin{equation}\widetilde{\nabla}:\Gamma(TM)\times {\mathcal{A}}^{\bullet}(M)\rightarrow A^1_d(M),~
(\widetilde{\nabla}_X\omega)(Y):=X(\omega(Y))-\left<\left<\widetilde{\nabla}^{\flat}_XY,\omega\right>\right>_\bullet.
\end{equation}
\end{defn}
Similarly to Theorem 6.13 in \cite{St1}, we have
\begin{prop}The almost product covariant derivative $\widetilde{\nabla}$ has the following properties:\\
\indent (1) Additivity and $\mathcal{R}$-linearity in each of its arguments.\\
\indent  (2) $\widetilde{\nabla}_{fX}\omega=f\widetilde{\nabla}_{X}\omega.$\\
\indent (3) $\widetilde{\nabla}_X(f\omega)=f\widetilde{\nabla}_X\omega+X(f)\omega.$\\
\indent (4)$\widetilde{\nabla}_{X}Y^\flat=\widetilde{\nabla}^\flat_{X}Y.$\\
\end{prop}
\begin{defn}
A singular almost product semi-Riemannian manifold $(M,g,J)$ satisfying $\widetilde{\nabla}^\flat_XY\in {\mathcal{A}}^{\bullet}(M)$ and $\widetilde{\nabla}_Z(\overline{\nabla}^\flat_XY)\in {\mathcal{A}}^{\bullet}(M)$ for any $X,Y,Z\in\Gamma(TM)$ is called a almost product semi-regular semi-Riemannian manifold.
\end{defn}
Similarly to Proposition 6.23 in \cite{St1}, we have
 \begin{prop} A almost product radical stationary semi-Riemannian manifold $(M,g,J)$ is a almost product semi-regular semi-Riemannian manifold if and only if $\widetilde{\mathcal{K}}(X,Y,\bullet)\widetilde{\mathcal{K}}(Z,T,\bullet)\in C^{\infty}(M).$
\end{prop}
\begin{ex} Let $(M,\widetilde{g},J)$ be a non-degenerate almost product semi-Riemannian manifold and $\Omega\in C^{\infty}(M)$ and $\Omega\geq 0$, then $(M,g=\Omega^2\widetilde{g},J)$ is a almost product semi-regular semi-Riemannian manifold. This is similar to the proof of Theorem 9.3 in \cite{St1}.
\end{ex}
On a almost product semi-regular semi-Riemannian manifold, we define the Riemann curvature tensor of the almost product
 covariant derivative
 \begin{equation}
\widetilde{R}(X,Y,Z,T):=(\widetilde{\nabla}_X\widetilde{\nabla}^\flat_YZ)(T)-(\widetilde{\nabla}_Y\widetilde{\nabla}^\flat_XZ)(T)
-(\widetilde{\nabla}^\flat_{[X,Y]}Z)(T).
\end{equation}
 Then similarly to Theorem 7.5 in \cite{St1}, we have
 \begin{thm} Let $(M,g)$ be a almost product semi-regular semi-Riemannian manifold, then $\widetilde{R}(X,Y,Z,T)$ is a smooth $(0,4)$-tensor field.
 \end{thm}
Similarly to Proposition 8.1 in \cite{St1}, we have
\begin{prop} For any vector fields $X,Y,Z,T\in\Gamma(TM)$ on a almost product semi-regular semi-Riemannian manifold $(M,g,J)$
 \begin{align}
\widetilde{R}(X,Y,Z,T)&=X((\widetilde{\nabla}^\flat_YZ)(T))-Y((\widetilde{\nabla}^\flat_XZ)(T))-(\widetilde{\nabla}^\flat_{[X,Y]}Z)(T)\\\notag
&+\left<\left<\widetilde{\nabla}^{\flat}_XZ,\widetilde{\nabla}^{\flat}_YT\right>\right>_\bullet
-\left<\left<\widetilde{\nabla}^{\flat}_YZ,\widetilde{\nabla}^{\flat}_XT\right>\right>_\bullet.
\end{align}
 \begin{align}
\widetilde{R}(X,Y,Z,T)&=X(\widetilde{\mathcal{K}}(Y,Z,T))-Y(\widetilde{\mathcal{K}}(X,Z,T))-\widetilde{\mathcal{K}}([X,Y],Z,T)\\\notag
&+\widetilde{\mathcal{K}}(X,Z,\bullet)\widetilde{\mathcal{K}}(Y,T,\bullet)-\widetilde{\mathcal{K}}(Y,Z,\bullet)\widetilde{\mathcal{K}}(X,T,\bullet).
\end{align}
\end{prop}
We let $J=(J_1,J_2):B\times_fF\rightarrow B\times_fF$ be the almost product structure on $B\times_fF$ and $J_1:TB\rightarrow TB$,
$J_2:TF\rightarrow TF$ and $J^2_1=id_{TB}$, $J^2_2=id_{TF}$.
By Proposition 3.6 in \cite{St2} and Definition 4.2, we have
\begin{prop} Let $B\times_fF$ be a degenerate warped product and let the vector fields $X,Y,Z\in\Gamma(TB)$ and $U,V,W\in\Gamma(TF)$. Let
$\widetilde{\mathcal{K}}$ be the almost product Koszul form on $B\times_fF$ and $\widetilde{\mathcal{K}}_B$, $\widetilde{\mathcal{K}}_F$
the lifts of the almost product Koszul form on $B$, respectively $F$. Then\\
\indent (1) $\widetilde{\mathcal{K}}(X,Y,Z)=\widetilde{\mathcal{K}}_B(X,Y,Z)$.\\
\indent  (2) $\widetilde{\mathcal{K}}(X,Y,W)=\widetilde{\mathcal{K}}(X,W,Y)=\widetilde{\mathcal{K}}(W,X,Y)=0.$\\
\indent (3) $\widetilde{\mathcal{K}}(X,V,W)=fg_F(V,W)X(f).$\\
\indent (4)$\widetilde{\mathcal{K}}(V,X,W)=-\widetilde{\mathcal{K}}(V,W,X)=\frac{1}{2}[fg_F(V,W)X(f)+f(JX)(f)g_F(V,JW)].$\\
\indent (5) $\widetilde{\mathcal{K}}(U,V,W)=f^2\widetilde{\mathcal{K}}_F(U,V,W).$\\
\end{prop}
\begin{thm}
Let $(B,g_B,J_1)$ and $(F,g_F,J_2)$ be two almost product radical-stationary semi-Riemannian manifolds and $f\in C^{\infty}(B)$ so that
$df\in {\mathcal{A}}^{\bullet}(B)$, Then the warped product manifold $(B\times _fF,J)$ is a almost product radical-stationary semi-Riemannian manifold.
\end{thm}
\begin{proof}
Since $df\in {\mathcal{A}}^{\bullet}(B)$, for $X\in\Gamma_0(TB)$, then $JX\in\Gamma_0(TB)$ and $(JX)(f)=(df)(JX)=0.$ Using Proposition 4.14 and similarly to the proof of Theorem 4.1 in \cite{St2}, we can prove this theorem.
\end{proof}
\begin{thm}
Let $(B,g_B,J_1)$ be a almost product non-degererate semi-Riemannian manifold and $(F,g_F,J_2)$ be a almost product semi-regular semi-Riemannian manifolds and $f\in C^{\infty}(B)$, Then the warped product manifold $(B\times _fF,J)$ is a almost product semi-regular semi-Riemannian manifold.
\end{thm}
\begin{proof}
By Theorem 4.15 and Proposition 4.10, we only show $\widetilde{\mathcal{K}}(X,Y,\bullet)\widetilde{\mathcal{K}}(Z,T,\bullet)\in C^{\infty}(M).$
We will denote the covariant contraction with $\bullet$ on $B\times_fF$, and with $\bullet^B$, $\bullet^F$ denote the covariant contraction on $B$, respectively $F$. Let $X_B,Y_B,Z_B,T_B\in\Gamma(TB)$, $X_F,Y_F,Z_F,T_F\in\Gamma(TF)$.  We assume that $f>0$, then
\begin{align}
\widetilde{\mathcal{K}}(X_B,Y_B,\bullet)\widetilde{\mathcal{K}}(Z_B,T_B,\bullet)&=
\widetilde{\mathcal{K}}(X_B,Y_B,\bullet^B)\widetilde{\mathcal{K}}(Z_B,T_B,\bullet^B)\\\notag
&+\frac{1}{f^2}\widetilde{\mathcal{K}}(X_B,Y_B,\bullet^F)\widetilde{\mathcal{K}}(Z_B,T_B,\bullet^F)\\\notag
&=\widetilde{\mathcal{K}}_B(X_B,Y_B,\bullet^B)\widetilde{\mathcal{K}}_B(Z_B,T_B,\bullet^B)\in C^{\infty}(B).
\end{align}
where we used (1) (2) in Proposition 4.14. By Proposition 4.14, we can get similarly
\begin{align}
\widetilde{\mathcal{K}}(X_B,Y_B,\bullet)\widetilde{\mathcal{K}}(Z_F,T_B,\bullet)=
\widetilde{\mathcal{K}}(X_B,Y_B,\bullet)\widetilde{\mathcal{K}}(T_B,Z_F,\bullet)=0,
\end{align}
\begin{align}
\widetilde{\mathcal{K}}(X_B,Y_B,\bullet)\widetilde{\mathcal{K}}(Z_F,T_F,\bullet)&=
-\frac{1}{2}fg^*_B(\widetilde{\nabla}^\flat_{X_B}Y_B,df)g_F(Z_F,T_F)\\\notag
&-\frac{1}{2}fg^*_B(\widetilde{\nabla}^\flat_{X_B}Y_B,(df)J)g_F(Z_F,JT_F).
\end{align}
\begin{align}
\widetilde{\mathcal{K}}(X_B,Y_F,\bullet)\widetilde{\mathcal{K}}(Z_B,T_F,\bullet)&=
X_B(f)Z_B(f)g_F(T_F,Y_F).
\end{align}
\begin{align}
\widetilde{\mathcal{K}}(X_B,Y_F,\bullet)\widetilde{\mathcal{K}}(T_F,Z_B,\bullet)&=
\frac{1}{2}X_B(f)Z_B(f)g_F(T_F,Y_F)\\\notag
&+\frac{1}{2}X_B(f)(JZ_B)(f)g_F(T_F,JY_F).
\end{align}
\begin{align}
\widetilde{\mathcal{K}}(Y_F,X_B,\bullet)\widetilde{\mathcal{K}}(T_F,Z_B,\bullet)&=
\frac{1}{4}X_B(f)Z_B(f)g_F(T_F,Y_F)\\\notag
&+\frac{1}{4}X_B(f)(JZ_B)(f)g_F(T_F,JY_F)\\\notag
&+\frac{1}{4}(JX_B)(f)Z_B(f)g_F(T_F,JY_F)\\\notag
&+\frac{1}{4}(JX_B)(f)(JZ_B)(f)g_F(T_F,Y_F).
\end{align}
\begin{align}
\widetilde{\mathcal{K}}(X_B,Y_F,\bullet)\widetilde{\mathcal{K}}(Z_F,T_F,\bullet)=
fX_B(f)\widetilde{\mathcal{K}}_F(Z_F,T_F,Y_F).
\end{align}
\begin{align}
\widetilde{\mathcal{K}}(Y_F,X_B,\bullet)\widetilde{\mathcal{K}}(Z_F,T_F,\bullet)&=
\frac{1}{2}fX_B(f)\widetilde{\mathcal{K}}_F(Z_F,T_F,Y_F)\\\notag
&+\frac{1}{2}f(JX_B)(f)\widetilde{\mathcal{K}}_F(Z_F,T_F,JY_F).
\end{align}
\begin{align}
\widetilde{\mathcal{K}}(X_F,Y_F,\bullet)\widetilde{\mathcal{K}}(Z_F,T_F,\bullet)&=
\frac{1}{4}f^2[g^*_B(df,df)g_F(X_F,Y_F)(Z_F,T_F)\\\notag
&+g^*_B(df\circ J,df)g_F(X_F,JY_F)(Z_F,T_F)\\\notag
&+g^*_B(df,df\circ J)g_F(X_F,Y_F)(Z_F,JT_F)\\\notag
&+g^*_B(df\circ J,df\circ J)g_F(X_F,JY_F)(Z_F,JT_F)]\\\notag
&+f^2\widetilde{\mathcal{K}}_F(X_F,Y_F,\bullet^F)\widetilde{\mathcal{K}}_F(Z_F,T_F,\bullet^F).
\end{align}
We know that (4.10)-(4.18) are smooth and smoothly extend to the region $f=0$, so $(B\times _fF,J)$ is a almost product semi-regular semi-Riemannian manifold.
\end{proof}
\begin{thm} Let $(B,g_B,J_1)$ be a nondegenerate almost product manifold and $(F,g_F,J_2)$ be a semi-regular almost product manifold and $f\in C^{\infty}(B)$. Let the vector fields $X,Y,Z,T\in\Gamma(TB)$ and $U,V,W,Q\in\Gamma(TF)$, then\\
\indent (1) $\widetilde{R}(X,Y,Z,T)=\widetilde{R}_B(X,Y,Z,T)$.\\
\indent  (2) $\widetilde{R}(X,Y,Z,Q)=\widetilde{R}(Z,Q,X,Y)=0.$\\
\indent (3)$\widetilde{R}(X,Y,W,Q)=\widetilde{R}(W,Q,X,Y)=0.$\\
\indent (4) $\widetilde{R}(X,V,W,T)=-\frac{1}{2}fH^f(X,T)g_F(V,W)-\frac{f}{2}[X(J_1T)(f)-g^*_B(\widetilde{\nabla}^\flat_XT,df\circ J_1)]g_F(V,JW).$\\
\indent (5) $\widetilde{R}(U,V,Z,Q)=\frac{1}{4}f Z(f)[-{\mathcal{K}}_F(V,Q,U)+{\mathcal{K}}_F(V,J Q,J U)+{\mathcal{K}}_F(U,Q,V)
-{\mathcal{K}}_F(U,J Q,J V)]$\\
\indent $-\frac{1}{4}f (J Z)(f)[{\mathcal{K}}_F(V,J Q,U)-{\mathcal{K}}_F(V, Q,J U)-{\mathcal{K}}_F(U,J Q,V)
+{\mathcal{K}}_F(U,Q,J V)].$\\
\indent (6)$\widetilde{R}(Z,Q,U,V)=0.$\\
\indent (7)  $\widetilde{R}(U,V,W,Q)=f^2\widetilde{R}_F(U,V,W,Q)+\frac{f^2}{4}\left\{g^*_B(df,df)
[g_F(U,W)g_F(V,Q)-g_F(V,W)g_F(U,Q)]\right.$\\
\indent $+g^*_B(df\circ J_1,df)
[g_F(U,JW)g_F(V,Q)+g_F(U,W)g_F(V,JQ)$\\
\indent $-g_F(V,JW)g_F(U,Q)-g_F(V,W)g_F(U,JQ)]$\\
\indent $\left.
+g^*_B(df\circ J_1,df\circ J_1)
[g_F(U,JW)g_F(V,JQ)-g_F(V,JW)g_F(U,JQ)]\right\}.
$
\end{thm}
\begin{proof}In order to prove these identities, we will use (4.9), Proposition 4.14, (4.10)-(4.18).
 \begin{align}
(1)\widetilde{R}(X,Y,Z,T)&=X(\widetilde{\mathcal{K}}(Y,Z,T))-Y(\widetilde{\mathcal{K}}(X,Z,T))-\widetilde{\mathcal{K}}([X,Y],Z,T)\\\notag
&+\widetilde{\mathcal{K}}(X,Z,\bullet)\widetilde{\mathcal{K}}(Y,T,\bullet)
-\widetilde{\mathcal{K}}(Y,Z,\bullet)\widetilde{\mathcal{K}}(X,T,\bullet)\\\notag
&=X(\widetilde{\mathcal{K}}_B(Y,Z,T))-Y(\widetilde{\mathcal{K}}_B(X,Z,T))-\widetilde{\mathcal{K}}_B([X,Y],Z,T)\\\notag
&+\widetilde{\mathcal{K}}_B(X,Z,\bullet^B)\widetilde{\mathcal{K}})_B(Y,T,\bullet^B)
-\widetilde{\mathcal{K}}_B(Y,Z,\bullet^B)\widetilde{\mathcal{K}}_B(X,T,\bullet^B)\\\notag
&=\widetilde{R}_B(X,Y,Z,T),
\end{align}
where we applied (1) from Proposition 4.14 and (4.10).
\begin{align}
(2)\widetilde{R}(X,Y,Z,Q)&=X(\widetilde{\mathcal{K}}(Y,Z,Q))-Y(\widetilde{\mathcal{K}}(X,Z,Q))-\widetilde{\mathcal{K}}([X,Y],Z,Q)\\\notag
&+\widetilde{\mathcal{K}}(X,Z,\bullet)\widetilde{\mathcal{K}}(Y,Q,\bullet)
-\widetilde{\mathcal{K}}(Y,Z,\bullet)\widetilde{\mathcal{K}}(X,Q,\bullet)\\\notag
&=0,
\end{align}
where we applied (2) from Proposition 4.14 and (4.11).
\begin{align}
(4)\widetilde{R}(X,V,W,T)&=X(\widetilde{\mathcal{K}}(V,W,T))-V(\widetilde{\mathcal{K}}(X,W,T))-\widetilde{\mathcal{K}}([X,V],W,T)\\\notag
&+\widetilde{\mathcal{K}}(X,W,\bullet)\widetilde{\mathcal{K}}(V,T,\bullet)
-\widetilde{\mathcal{K}}(V,W,\bullet)\widetilde{\mathcal{K}}(X,T,\bullet)\\\notag
&=X[-\frac{1}{2}fg_F(V,W)T(f)-\frac{f}{2}(JT)(f)g_F(V,JW)]\\\notag
&+\frac{1}{2}T(f)X(f)g_F(V,W)+\frac{1}{2}(JT)(f)X(f)g_F(V,JW)\\\notag
&+\frac{1}{2}fg^*_B(\widetilde{\nabla}^\flat_XT,df)g_F(V,W)
+\frac{1}{2}fg^*_B(\widetilde{\nabla}^\flat_XT,df\circ J_1)g_F(V,JW)\\\notag
&=-\frac{1}{2}fH^f(X,T)g_F(V,W)-\frac{f}{2}[X(J_1T)(f)-g^*_B(\widetilde{\nabla}^\flat_XT,df\circ J_1)]g_F(V,JW),
\end{align}
where we applied (2) and (4) from Proposition 4.14 and (4.12),(4.14). Similarly we can prove other equalities.
\end{proof}

\section{Degenerate multiply warped products of singular semi-Riemannian manifolds}
\begin{defn}
 Let $(B,g_B)$ and $(F_j,g_{F_J})$ $1\leq j\leq l$ be singular
semi-Riemannian manifolds, and $f_j\in C^{\infty}(B)$ $1\leq j\leq l$ smooth functions. The multiply warped product
of $B$ and $F_j$ with warping function $f_j$, $1\leq j\leq l$ is the semi-Riemannian manifold
\begin{equation}
B\times F_1\times \cdots \times F_l:=(B\times F_1\times \cdots \times F_l,\pi^*_B(g_B)+\sum_{j=1}^l(f_j\circ \pi_B)^2\pi^*_{F_j}(g_{F_j})),
\end{equation}
where $\pi_B: B \times F_1\times \cdots \times F_l \rightarrow B $ and  $\pi_{F_j}: B\times F_1\times \cdots \times F_l \rightarrow F_j $ are the canonical projections.
\end{defn}
\indent The inner product on
$ B \times F_1\times \cdots \times F_l $ takes, for any point $p\in B \times F_1\times \cdots \times F_l $ and for any pair of tangent vectors $x,y\in T_p(B \times F_1\times \cdots \times F_l )$, the explicit form
\begin{equation}
\left<x,y\right>=\left<d\pi_B(x),d\pi_B(y)\right>_B+\sum_{j=1}^lf_j^2(p)\left<d\pi_{F_j}(x),d\pi_{F_j}(y)\right>_{F_j}.
\end{equation}
Similarly to Proposition 3.6 in \cite{St2}, we have
\begin{prop} Let $B \times F_1\times \cdots \times F_l$ be a degenerate multiply warped product and let the vector fields $X,Y,Z\in\Gamma(TB)$ and $U_j,V_j,W_j\in\Gamma(TF_j)$. Then\\
\indent (1) ${\mathcal{K}}(X,Y,Z)={\mathcal{K}}_B(X,Y,Z)$.\\
\indent  (2) ${\mathcal{K}}(X,Y,W_j)={\mathcal{K}}(X,W_j,Y)={\mathcal{K}}(W_j,X,Y)=0.$\\
\indent (3) ${\mathcal{K}}(X,V_j,W_j)={\mathcal{K}}(V_j,X,W_j)=-{\mathcal{K}}(V_j,W_j,X)=f_jg_{F_j}(V_j,W_j)X(f_j).$\\
\indent (4)${\mathcal{K}}(X,V_i,W_j)={\mathcal{K}}(V_i,X,W_j)={\mathcal{K}}(V_i,W_j,X)=0$ for $i\neq j$.\\
\indent (5) ${\mathcal{K}}(U_j,V_j,W_j)=f_j^2{\mathcal{K}}_{F_j}(U_j,V_j,W_j).$\\
\indent (6) ${\mathcal{K}}(U_i,V_j,W_k)=0$ where $i,j,k$ are different. \\
\indent (7) ${\mathcal{K}}(U_i,V_i,W_j)={\mathcal{K}}(U_i,W_j,V_i)={\mathcal{K}}(W_j,U_i,V_i)=0$ for $i\neq j$.\\
\end{prop}
By Proposition 5.2, similarly to Theorem 4.1 in \cite{St2}, we can get
\begin{thm}
Let $(B,g_B)$ and $(F_j,g_{F_j})$, $1\leq j\leq l$ be radical-stationary semi-Riemannian manifolds and $f_j\in C^{\infty}(B)$ so that
$df_j\in {\mathcal{A}}^{\bullet}(B)$, Then the multiply warped product manifold $B \times F_1\times \cdots \times F_l$ is a radical-stationary semi-Riemannian manifold.
\end{thm}
\begin{thm}
Let $(B,g_B)$ be a non-degererate semi-Riemannian manifold and $(F_j,g_{F_j})$, $1\leq j\leq l$ be semi-regular semi-Riemannian manifolds and $f_j\in C^{\infty}(B)$, Then the multiply warped product manifold $B \times F_1\times \cdots \times F_l$ is a semi-regular semi-Riemannian manifold.
\end{thm}
\begin{proof}
By Proposition 2.10 in \cite{St2}, we only show ${\mathcal{K}}(X,Y,\bullet){\mathcal{K}}(Z,T,\bullet)\in C^{\infty}(B \times F_1\times \cdots \times F_l).$ Let $X_B,Y_B,Z_B,T_B\in\Gamma(TB)$, $X_{F_j},Y_{F_j},Z_{F_j},T_{F_j}\in\Gamma(T{F_j})$.  We assume that $f_j>0$. Then by Proposition 5.2, we have
\begin{align}
{\mathcal{K}}(X_B,Y_B,\bullet){\mathcal{K}}(Z_B,T_B,\bullet)={\mathcal{K}}_B(X_B,Y_B,\bullet^B){\mathcal{K}}_B(Z_B,T_B,\bullet^B)\in C^{\infty}(B).
\end{align}
\begin{align}
{\mathcal{K}}(X_B,Y_B,\bullet){\mathcal{K}}(Z_{F_j},T_B,\bullet)=
{\mathcal{K}}(X_B,Y_B,\bullet){\mathcal{K}}(T_B,Z_{F_j},\bullet)=0,
\end{align}
\begin{align}
{\mathcal{K}}(X_B,Y_B,\bullet){\mathcal{K}}(Z_{F_j},T_{F_j},\bullet)&=
-f_j(\nabla^B_{X_B}Y_B)(f_j)g_{F_j}(Z_{F_j},T_{F_j}).
\end{align}
\begin{align}
{\mathcal{K}}(X_B,Y_{F_j},\bullet){\mathcal{K}}(Z_B,T_{F_j},\bullet)&={\mathcal{K}}(X_B,Y_{F_j},\bullet){\mathcal{K}}(T_{F_j},Z_B,\bullet)\\\notag
&=X_B(f_j)Z_B(f_j)g_{F_j}(T_{F_j},Y_{F_j}).
\end{align}
\begin{align}
{\mathcal{K}}(X_B,Y_{B},\bullet){\mathcal{K}}(Z_{F_i},T_{F_j},\bullet)={\mathcal{K}}(X_B,Y_{F_j},\bullet){\mathcal{K}}(Z_{B},T_{F_i},\bullet)=0,
\end{align}
\indent \indent \indent where $i\neq j$.
\begin{align}
{\mathcal{K}}(X_B,Y_{F_j},\bullet){\mathcal{K}}(Z_{F_j},T_{F_j},\bullet)=f_jX_B(f_j){\mathcal{K}}_{F_j}(Z_{F_j},T_{F_j},Y_{F_j}).
\end{align}
\begin{align}
{\mathcal{K}}(X_B,Y_{F_i},\bullet){\mathcal{K}}(Z_{F_j},T_{F_j},\bullet)={\mathcal{K}}(X_B,Y_{F_j},\bullet){\mathcal{K}}(Z_{F_i},T_{F_j},\bullet)
=0,
\end{align}
\indent \indent \indent where $i\neq j$.
\begin{align}
{\mathcal{K}}(X_B,Y_{F_j},\bullet){\mathcal{K}}(Z_{F_k},T_{F_\alpha},\bullet)
=0,
\end{align}
\indent \indent \indent where $j,k,\alpha$ are different.
\begin{align}
{\mathcal{K}}(X_{F_j},Y_{F_j},\bullet){\mathcal{K}}(Z_{F_j},T_{F_j},\bullet)&=
f_j^2g^*_B(df_j,df_j)g_{F_j}(X_{F_j},Y_{F_j})g_{F_j}(Z_{F_j},T_{F_j})\\\notag
&+f_j^2{\mathcal{K}}_{F_j}(X_{F_j},Y_{F_j},\bullet^{F_j}){\mathcal{K}}_{F_j}(Z_{F_j},T_{F_j},\bullet^{F_j}).
\end{align}
\begin{align}
{\mathcal{K}}(X_{F_i},Y_{F_j},\bullet){\mathcal{K}}(Z_{F_j},T_{F_j},\bullet)=0,
\end{align}
\indent \indent \indent where $i\neq j$.
\begin{align}
{\mathcal{K}}(X_{F_i},Y_{F_i},\bullet){\mathcal{K}}(Z_{F_j},T_{F_j},\bullet)&=
f_if_jg^*_B(df_i,df_j)g_{F_i}(X_{F_i},Y_{F_i})g_{F_j}(Z_{F_j},T_{F_j}),
\end{align}
\indent \indent \indent where $i\neq j$.

\begin{align}
{\mathcal{K}}(X_{F_i},Y_{F_j},\bullet){\mathcal{K}}(Z_{F_i},T_{F_j},\bullet)=0,
\end{align}
\indent \indent \indent where $i\neq j$.

\begin{align}
{\mathcal{K}}(X_{F_i},Y_{F_i},\bullet){\mathcal{K}}(Z_{F_j},T_{F_k},\bullet)=0,
\end{align}
\indent \indent \indent where $i,j,k$ are different.

\begin{align}
{\mathcal{K}}(X_{F_i},Y_{F_j},\bullet){\mathcal{K}}(Z_{F_k},T_{F_\alpha},\bullet)=0,
\end{align}
\indent \indent \indent where $i,j,k,\alpha$ are different.\\
\indent We know that (5.3)-(5.16) are smooth and smoothly extend to the region $f_j=0$, so $B \times F_1\times \cdots \times F_l$ is a semi-regular semi-Riemannian manifold.
\end{proof}
Using Proposition 2.13 in \cite{St2} and Proposition 5.2 and (5.3)-(5.16) and similarly to Theorem 5.2 in \cite{St2}, we have

\begin{thm} Let $(B,g_B)$ be a non-degererate semi-Riemannian manifold and $(F_j,g_{F_j})$, $1\leq j\leq l$ be semi-regular semi-Riemannian manifolds and $f_j\in C^{\infty}(B)$. Let $X_B,Y_B,Z_B,T_B\in\Gamma(TB)$, $X_{F_j},Y_{F_j},Z_{F_j},T_{F_j}\in\Gamma(T{F_j})$. Then\\
\indent (1) ${R}(X,Y,Z,T)={R}_B(X,Y,Z,T)$.\\
\indent  (2) ${R}(X,Y,Z,Q_j)={R}(X,Y,W_j,Q_j)=0.$\\
\indent (3) ${R}(X,V_j,W_j,T)=-f_jH^{f_j}(X,T)g_{F_j}(V_j,W_j).$\\
\indent (4) ${R}(X,Y,W_j,Q_k)=R(X,V_j,W_k,T)=R(X,V_j,W_k,Q_j)=R(X,W_k,V_j,Q_j)=0,$ for $j\neq k$.\\
\indent (5) $R(U_j,V_j,Z,Q_j)=0.$\\
\indent (6) $R(X,W_k,V_j,Q_\alpha)=0$, where $k,j,\alpha$ are different.\\
\indent (7)  ${R}(U_j,V_j,W_j,Q_j)=f_j^2{R}_{F_j}(U_j,V_j,W_j,Q_j)+f_j^2g^*_B(df_j,df_j)[g_{F_j}(U_j,W_j)g_{F_j}(V_j,Q_j)
-g_{F_j}(V_j,W_j)g_{F_j}(U_j,Q_j)$.\\
\indent (8) ${R}(U_k,V_j,W_j,Q_j)={R}(U_k,V_k,W_j,Q_j)=0,$ for $j\neq k$.\\
\indent (9) ${R}(U_k,V_j,W_k,Q_j)=f_kf_jg^*_B(df_k,df_j)g_{F_k}(U_k,W_k)g_{F_j}(V_j,Q_j),$ for $j\neq k$.\\
\indent (10) ${R}(U_j,V_j,W_k,Q_\alpha)={R}(U_j,V_k,W_j,Q_\alpha)=0$, where $k,j,\alpha$ are different.\\
\indent (11) ${R}(U_j,V_k,W_\alpha,Q_\beta)=0$, where $k,j,\alpha,\beta$ are different.\\
\end{thm}

\section{Acknowledgements}

The author was supported in part by  NSFC No.11771070. The author thanks the referee for his (or her) careful reading and helpful comments.

\vskip 1 true cm


\bigskip
\bigskip

\noindent {\footnotesize {\it Y. Wang} \\
{School of Mathematics and Statistics, Northeast Normal University, Changchun 130024, China}\\
{Email: wangy581@nenu.edu.cn}

\end{document}